\documentclass{amsart}
\usepackage{amsfonts}
\usepackage{graphicx}
\usepackage{amsmath}
\usepackage{amssymb}
\usepackage{amsthm}
\usepackage{color}
\usepackage{cancel}
\usepackage{enumerate}
\usepackage{tikz}
\tikzset{vertex/.style={circle,fill,draw,scale=.2}}
\usepackage[all, cmtip]{xy} 
\newtheorem{theorem}{Theorem}

\newtheorem{corollary}[theorem]{Corollary}

\newtheorem{example}[theorem]{Example}

\newtheorem{lemma}[theorem]{Lemma}

\newtheorem{proposition}[theorem]{Proposition}

\begin{document}
\title{An introduction to a supersymmetric graph algebra}
\author{Katherine Radler}
\address{Department of Mathematics and Statistics, St. Louis University, St.
Louis, MO-63103, USA} 
\email{katie.radler@slu.edu}
\author{Ashish K. Srivastava}
\address{Department of Mathematics and Statistics, St. Louis University, St.
Louis, MO-63103, USA} 
\email{ashish.srivastava@slu.edu}
\thanks{The work of the second author is partially supported by a grant from Simons Foundation (grant number 426367).}
\dedicatory{Dedicated to the loving memory of Anupam Srivastava}
\begin{abstract}
In this paper we propose a graph superalgebra which is the supersymmetric analogue of Leavitt path algebras. We find a basis for these superalgebras and characterize when they have polynomial growth.  
\end{abstract}

\maketitle

\bigskip

\section{Introduction}

\bigskip

\noindent The objective behind this paper is to introduce the notion of a graph superalgebra which is the supersymmetric analogue of Leavitt path algebras. The inspiration comes from \cite{C} where Cirafici associates the line defect with a certain path on the quiver and gives rules to compute its framed spectrum. It is known that limit points of the compatification of Teichm\"{u}ller spaces correspond to tropical solutions of the Ptolemy relations. Integral limit points correspond to simple closed curves, but this only gives countably many limit points. The real limit points correspond to laminations. In \cite{C}, lamination on a triangulated  curve $C$ is described algebraically on the extended BPS quiver $\overline{Q}$, where additional vertices are inserted that correspond to boundary segments. We take the example below from \cite{C}. 

Consider the following operator in the Argyres-Douglas $A_2$ theory, drawn as paths on the extended BPS quiver. 

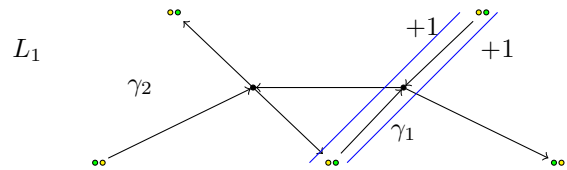
\begin{center}
\begin{tikzpicture}

\node[vertex, fill = green] (u1g) at (-2, 1) {};
\node (u1) at (-2.05, 1) {}; 
\node[vertex, fill = yellow] (u1y) at (-2.1, 1) {}; 
\node[vertex, fill = yellow] (u2y) at (2, 1) {}; 
\node (u2) at (2.05, 1) {}; 
\node[vertex, fill = green] (u2g) at (2.1, 1) {};
\node[vertex] (v1) at (-1,0) {}; 
\node[vertex] (v2) at (1,0) {}; 
\node[vertex, fill = green] (w1g) at (-3.1, -1) {}; 
\node (w1) at (-3.05, -1) {}; 
\node[vertex, fill = yellow] (w1y) at (-3, -1) {}; 
\node[vertex, fill = green] (w2g) at (0.1, -1) {}; 
\node (w2) at (0.05, -1) {}; 
\node[vertex, fill = yellow] (w2y) at (0,-1) {}; 
\node[vertex, fill = green] (w3g) at (3, -1) {}; 
\node (w3) at (3.05, -1) {}; 
\node[vertex, fill = yellow] (w3y) at (3.1, -1) {};  
\path[->] (v1) edge (u1);
\path[->] (w1) edge (v1); 
\path[->] (v1) edge (w2); 
\path[->] (w2) edge (v2); 
\path[->] (v2) edge (v1); 
\path[->] (u2) edge (v2); 
\path[->] (v2) edge (w3); 
\node (L1) at (-4, .5) {$L_1$};
\node (g2) at (-2.5, 0) {$\gamma_2$}; 
\path[-, color = blue] (1.75, 1) edge (-.25, -1); 
\path[-, color = blue] (2.25, 1) edge (.25, -1); 
\node (1) at (1.25, .8) {$+1$};
\node (12) at (2.25, .5) {$+1$}; 
\node (g1) at (1, -.6) {$\gamma_1$}; 
\end{tikzpicture}
\end{center} 

\noindent Here the boundary vertices are split in green and yellow vertices corresponding to two stokes sectors on the boundary segment separated by a special point. The paths in $L_1$ are given by formal strings of elements  
\[ v_g e_{G\gamma_1} v_{\gamma_1} e^*_{y \gamma_1} v_y \] 
\[ v_\gamma e_{y \gamma_1} v_{\gamma_1} e^*_{G\gamma_1} v_G \] 
Clearly these may be viewed as monomials in Leavitt path superalgebra over extended BPS quiver $\overline{Q}$ by taking the vertices $v_G, v_Y$ and $v_{\gamma_1}$ as fermionic vertices. We believe that the study of Leavitt path superalgebras will open up the possibility of an algebraic characterization of physical line defect.  Furthermore, using the theory of cluster algebras, Cirafici derives an algorithm to compute the framed BPS spectra of new defects from known ones. Recently, in \cite{LMRS}, the supersymmetric version of cluster algebra was introduced. We hope that by developing the notion of supersymmetric analogue of Leavitt path algebra, we will be able to understand line defects in a much better way and we will be able to relate it with supersymmetric cluster algebras. 

Recall that superalgebras were introduced by physicists to provide an algebraic framework for describing the symmetry that arises when we deal with both types of elementary particles, Bosons and Fermions. Bosons are the particles that follow integral spin and Fermions follow half-integral spin. The Boson-Fermion symmetry is called supersymmetry and it holds the key to unified field theory.  We refer the reader to \cite{V} for more details on supersymmetry. A super vector space $V$ is a vector space that is $\mathbb Z_2$-graded, that is, it has a decomposition $V=V_0 \oplus V_1$ with $0, 1\in \mathbb Z_2:=\mathbb Z/2\mathbb Z$. The elements of $V_0$ are called the {\it even} (or Bosonic) elements and the elements of $V_1$ are called the {\it odd} (or Fermionic) elements. The elements in $V_0\cup V_1$ are called {\it homogeneous} and their {\it parity}, denoted by $p$, is defined to be $0$ or $1$ according as they are even or odd. The morphisms in the category of super vector spaces are linear maps which preserve the gradings. A superalgebra $A$ is an algebra such that the multiplication map $A\otimes A\rightarrow A$ is a morphism in the category of super vector spaces. This is the same as requiring $p(ab)=p(a)+p(b)$ for any two homogeneous elements $a$ and $b$ in $A$. A superalgebra $A$ is called {\it supercommutative} if $ab=(-1)^{p(a)p(b)}ba$ for any two homogeneous elements $a$ and $b$ in $A$. We refer the reader to \cite{Kac} and \cite{Manin} for more details on superalgebras.     

Leavitt path algebras are algebraic analogues of graph C*-algebras and are also natural generalizations of Leavitt algebras of type $($$1, n$$)$ constructed in \cite{Leavitt}. Let $K$ be a field and $E$ be an arbitrary directed graph. Let $E^0$ be the set of vertices, and $E^1$ be the set of edges of directed graph $E$. Consider two maps $r: E^1 \rightarrow E^0$ and $s:E^1 \rightarrow E^0$. For any edge $e$ in $E^1$, $s(e)$ is called the {\it source} of $e$ and $r(e)$ is called the {\it range} of $e$. If $e$ is an edge starting from vertex $v$ and pointing toward vertex $w$, then we imagine an edge starting from vertex $w$ and pointing toward vertex $v$ and call it the {\it ghost edge} of $e$ and denote it by $e^*$. We denote by $(E^1)^*$, the set of all ghost edges of directed graph $E$. If $v \in E^0$ does not emit any edges, i.e. $s^{-1}(v) = \emptyset$, then $v$ is called a {\it sink} and if $v$ emits an infinite number of edges, i.e. $|s^{-1}(v)| = \infty$, then $v$ is called an {\it infinite emitter}. If a vertex $v$ is neither a sink nor an infinite emitter, then $v$ is called a {\it regular vertex}. 

\bigskip

\noindent The {\it Leavitt path algebra} of $E$ with coefficients in $K$, denoted by $L_K(E)$, is the $K$-algebra generated by the sets $E^0$, $E^1$, and $(E^1)^*$, subject to the following conditions:
	\begin{enumerate}
		\item[(A1)] $v_iv_j = \delta_{ij} v_i$ for all $v_i, v_j \in E^0$.
		\item[(A2)] $s(e)e = e = er(e)$ and $r(e)e^* = e^* = e^*s(e)$ for all $e$ in $E^1$.
		\item[(CK1)] $e_i^*e_j = \delta_{ij} r(e_i)$ for all $e_i, e_j \in E^1$.
		\item[(CK2)]  If $v \in E^0$ is any regular vertex, then $v = \sum_{\{e \in E^1: s(e) = v\}} ee^*$.
	\end{enumerate}
Conditions (CK1) and (CK2) are known as the {\it Cuntz-Krieger relations}. If $E^0$ is finite, then $\sum\limits_{v_i \in E^0} v_i$ is an identity for $L_K(E)$ and if $E^0$ is infinite, then $E^0$ generates a set of local units for $L_K(E)$. See \cite{book} for further details on Leavitt path algebras. 	 

\bigskip

\noindent In this paper we first define Leavitt path superalgebras and then describe their monomials. Using this description we obtain Gr\"{o}bner-Shirsov basis for Leavitt path superalgebras and characterize when they have polynomial growth. We also describe the Jacobson radical of these superalgebras and characterize when they are von Neumann regular.   

\bigskip
    
\section{Leavitt path superalgebra}

\bigskip

\noindent Let $K$ be a field of characteristic different from 2. Consider a directed graph $E$ with two types of vertices: Bosonic vertices and Fermionic vertices. Let $B$ be the set of Bosonic vertices and $F$ be the set of Fermionic vertices. Let us denote the set of all vertices by $E^0$, and the set of edges by $E^1$. Note that $E^0=B\cup F$ and $B\cap F=\emptyset$. For any edge $e$ in $E^1$, $s(e)$ is called the {\it source} of $e$ and $r(e)$ is called the {\it range} of $e$. If $e$ is an edge starting from vertex $v$ and pointing toward vertex $w$, then we imagine an edge starting from vertex $w$ and pointing toward vertex $v$ and call it the {\it ghost edge} of $e$ and denote it by $e^*$. We denote by $(E^1)^*$, the set of all ghost edges of directed graph $E$. 

We will call an edge $e$, a {\it Bosonic (resp., Fermionic)} edge if both $s(e)$ and $r(e)$ are Bosonic (resp., Fermionic). For the convenience of readers, we will generally denote Bosonic edges with $e$ and Fermionic edges with $f$. We will call an edge $e$, a {\it mixed type edge} if one of the $s(e)$ and $r(e)$ is Bosonic and the other is Fermionic. Specifically, if $s(e)$ is Fermionic and $r(e)$ is Bosonic for an edge $e$, then we will call it {\it left Fermionic} and denote it by $f_l$. If for an edge $e$, $s(e)$ is Bosonic and $r(e)$ is Fermionic we will say it is {\it right Fermionic} and denote it by $f_r$. Consider a path $\alpha=e_1e_2 \ldots e_{n-1} e_{n} \ldots $ with $n \ge 2$ such that $r(e_i)=s(e_{i+1})$. We will call $\alpha$ to be a Bosonic path if $s(\alpha), r(\alpha) \in B$ and no vertices of $\alpha$ are Fermionic. We call a path $\alpha$ to be left Fermionic if $s(\alpha) \in F$ and no other vertices of $\alpha$ are Fermionic and right Fermionic if $r(\alpha) \in F$ and no other vertices of $\alpha$ are Fermionic. If both $s(\alpha)$ and $r(\alpha)$ are Fermionic and no other vertex of $\alpha$ is Fermionic, then we call $\alpha$ to be a Fermionic path.

Consider the $K$-algebra $SL_K(E)$ generated by $E^0$, $E^1$ and $(E^1)^*$ subject to the following relations:

\begin{enumerate}

\item 
\noindent \subitem (i) $v_iv_j=\delta_{ij}v_i$ for all $v_i, v_j \in B$.
\subitem (ii) $v_iv_j=-v_jv_i$ for all $v_i, v_j \in F$. In particular, $v_i^2=0$ for any $v_i\in F$.
\subitem (iii) $v_iv_j=v_jv_i$ if $v_i \in B$ and $v_j \in F$.

\item
\noindent \subitem (i) $s(e)e=e=er(e)$, if $s(e), r(e) \in B$.
\subitem (ii) $s(e)e=e$, $er(e)=0$, if $s(e) \in B$ and $r(e) \in F$.
\subitem (iii) $s(e)e=0$, $er(e)=e$, if $s(e) \in F$ and $r(e) \in B$.
\subitem (iv) $s(e)e=0=er(e)$, if $s(e), r(e) \in F$.

\item
\noindent \subitem (i) $r(e)e^*=e^*=e^*s(e)$, if $s(e), r(e) \in B$.
\subitem (ii) $r(e)e^*=e^*$, $e^*s(e)=0$, if $r(e) \in B$ and $s(e) \in F$.
\subitem (iii) $r(e)e^*=0$, $e^*s(e)=e^*$, if $r(e) \in F$ and $s(e) \in B$.
\subitem (iv) $r(e)e^*=0=e^*s(e)$, if $s(e), r(e) \in F$.

\item 
\noindent \subitem (i) $e^*e=r(e)$, if $r(e), s(e) \in B$.
\subitem (ii) $e^*e=0$, if $r(e) \in F$ or $s(e) \in F$.
\subitem (iii) $e_i^*e_j=0$, if $e_i \neq e_j$ with both $e_i$ and $e_j$ being Bosonic edges.

\medskip

\item   If $v \in B$ is any regular vertex, then $v = \sum_{\{e \in E^1: s(e) = v, r(e) \in B\}} ee^*$.

\end{enumerate}  

\bigskip

\noindent We multiply two edges $e$ and $f$ as $ef = e(r(e)s(f))f$. We multiply an edge $e$ and a vertex $v$ as $ve = (v s(e))e$ and $ev = e(r(e)v)$.  We will multiply any two words $x=x_1x_2\cdots x_n$ and $y = y_1y_2 \cdots y_m$ together as $xy = x_1\cdots x_n y_1 \cdots y_m$. This multiplication is consistent with multiplication in a Leavitt path algebra.\\

\noindent Define the parity $p:E^0\rightarrow \mathbb Z_2:=\mathbb Z/2\mathbb Z$ as $p(v)=0$ if $v \in B$ and $p(v)=1$ if $v \in F$. For any $e\in E^1$, define $p(e)=p(s(e))+p(r(e))=p(e^*)$. Let us denote the extended graph $(E^0, E^1\cup (E^1)^*)$ by $\overline{E}$. Consider the path algebra $K \overline{E}$ which is a $K$-algebra with paths of $\overline{E}$ as a basis and multiplication given by concatenation of paths. Our grading makes the path algebra $K\overline{E}$, a $\mathbb Z_2$-graded algebra. Since the defining relations above for $SL_K(E)$ are homogeneous with respect to this grading, we have that $SL_K(E)$ is also a $\mathbb Z_2$-graded algebra. This associative algebra $SL_K(E)$ is a supersymmetric analogue of Leavitt path algebra and so we call it Leavitt path superalgebra. 

\bigskip

\begin{example} 
Let $E$ be the graph with $n$ Bosonic vertices $\{v_1, ..., v_n\}$ and $m$ Fermionic vertices $\{w_1, ..., w_m\}$ with no edges. Then $SL_K(E) =A \otimes K[w_1, \cdots w_m]$ where $A=Kv_1+\ldots + Kv_n\cong K^n$ and $K[w_1, \ldots, w_m]$ is the Grassmann algebra generated by $w_1, \ldots, w_m$ with $w_iw_j = - w_jw_i$. 
\end{example}

\begin{example} 
Let $E$ be the graph below with $n$ Fermionic vertices and $m$ Bosonic vertices as shown below. Then $SL_K(E) \cong M_m(K) \otimes K[w_1, w_2, \cdots w_n]$, where $K[w_1, w_2, \cdots w_n]$ is the grassman algebra with $w_i^2 =0$. 
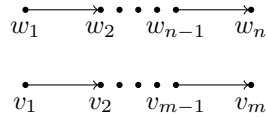
\begin{center} 
\begin{tikzpicture} 
\node[vertex, label = {below: $w_1$}] (w1) at (-2,1) {};
\node[vertex, label = {below: $w_2$}] (w2) at (-1, 1) {}; 
\node[vertex, label = {below: $w_{n-1}$}] (w3) at (0,1) {};
\node[vertex, label = {below: $w_{n}$}] (w4) at (1,1) {};
\node[vertex, label= {}] (dot) at (-.25, 1) {};
\node[vertex, label= {}] (dot) at (-.5, 1) {};
\node[vertex, label= {}] (dot) at (-.75, 1) {};
\path[->] (w1) edge (w2);
\path[->] (w3) edge (w4);  
\node[vertex, label = {below: $v_1$}] (v1) at (-2,0) {};
\node[vertex, label = {below: $v_2$}] (v2) at (-1, 0) {}; 
\node[vertex, label = {below: $v_{m-1}$}] (v3) at (0,0) {};
\node[vertex, label = {below: $v_m$}] (v4) at (1,0) {};
\node[vertex, label= {}] (dot) at (-.25, 0) {};
\node[vertex, label= {}] (dot) at (-.5, 0) {};
\node[vertex, label= {}] (dot) at (-.75, 0) {};
\path[->] (v1) edge (v2);
\path[->] (v3) edge (v4);  
\end{tikzpicture} 
\end{center}
\end{example}

\bigskip

\section{Monomials of Leavitt path superalgebras}

\bigskip

\noindent In this section we will provide description for monomials of Leavitt path superalgebras. We first prove a useful lemma.

\begin{lemma}
Let $x$ and $f$ be two elements of a Leavitt path superalgebra $SL_K(E)$ such that $s(f) \in F$. Then $xf = 0 = f^*x$. Similarly, if $r(f) \in F$ then $fx = 0 = x f^*$. 

Specifically, if $f$ is a Fermionic path or edge then $xf = 0 = fx$; if $f_l$ is a left Fermionic path or edge $xf_l = 0 = (f_l)^* x$; and if $f_r$ is a right Fermionic path or edge then $f_rx = 0 = x (f_r)^*$.
\end{lemma} 
 
\begin{proof}
Let $x$ be an arbitrary element of $SL_K(E)$. Then $x = \sum_{i=1}^n x_i$ where $x_i = x_i^1 \cdots x_i^m$ is a monomial. Now suppose $s(f) \in F$. Then 
\[ xf = \sum_{i=1}^n x_if = \sum_{i=1}^n x_i^1\cdots x_i^mf = \sum_{i=1}^n x_i^1\cdots x_i^mr(x_i^m)s(f)f = 0 \] 
Similarly we can show that if $r(f) \in F$ then $fx = 0$. 
\end{proof}

As a consequence we have the following useful corollary.

\begin{corollary}
The Leavitt path superalgebra $SL_K(E)$ has a multiplicative identity if and only if $E^0$ consists only of Bosonic vertices and is a finite set. In this case, $SL_K(E)=L_K(E)$ and $1=v_1+\cdots+v_n$ where $E^0=B=\{v_1, \ldots, v_n\}$.   
\end{corollary}

\begin{proof}
As we have noticed in the lemma above, if $f$ is any element in $SL_K(E)$ with $s(f)\in F$ or $r(f)\in F$ then $xf=0$ or $fx=0$ for any element $x\in SL_K(E)$. Thus it is clear that $SL_K(E)$ has a multiplicative identity if and only if there are no Fermionic vertices and $E^0$ consists only of Bosonic vertices. In this case $SL_K(E)=L_K(E)$. From the theory of Leavitt path algebras we know that in this case $E^0$ must be a finite set and $1=v_1+\cdots+v_n$ where $E^0=B=\{v_1, \ldots, v_n\}$.
\end{proof}

In the notation below, when we say that $p = \emptyset$ this implies that $qp = q$ or $pq = q$. We define $s(\alpha\beta^*) = s(\alpha)$ if $\alpha \neq \emptyset$ or $s(\beta^*)$ if $\alpha = \emptyset$. We define $r(\alpha\beta^*) = r(\beta^*)$ if $\beta \neq \emptyset$ or $r(\alpha)$ if $\beta = \emptyset$. 

Now we are ready to prove our main result of this section that describes monomials of Leavitt path superalgebras. 

\begin{theorem}\label{Monomials}
Each element $x$ of a Leavitt path superalgebra $SL_K(E)$ can be written as a linear combination of monomials of the following form. 
\[ x^0 w^1 x^1 w^2 x^2 \cdots x^{n-1}w^n x^n \] 
subject to the following conditions: 
\begin{enumerate}
\item $x^0 = \emptyset$ or $x^0 = \alpha_0\beta_0^*$ such that $r(x^0) \in B$ if $w^1 \neq \emptyset$
\item $w^i = w_{i,1}w_{i,2} \cdots w_{i,l}$ where $w_{i,l} \in F$ $w_{i,j} \neq w_{i,k}$ if $j \neq k$
\item If $i \in \{1, ..., n-1\}$ then $x^i = \alpha_i\beta_i^*$ such that $r(x^i), s(x^i) \in B$
\item $x^n = \emptyset$ or $x^n = \alpha_n\beta_n^*$ such that $s(x^n) \in B$. 
\item For $i \in \{0,1,..., n-1\}$, $r(x^i) = s(x^{i+1})$. 
\item For $i \in \{0,1,...,n\}$, $r(\alpha_i) = r(\beta_i)$. 
\end{enumerate} 
\end{theorem} 

\begin{proof}
We consider a monomial as a word consisting of the generators of $SL_K(E)$:  $\{w, v, e, e^*, f, f^*, f_l, f_l^*, f_r, f_r^*\}$. We will prove the result by induction on the length of the word $x$, denoted $l(x)$. 

 Suppose $x$ is a word with $l(x)=1$. Then $x$ must look like one of the following $\{w, v, e, e^*, f, f^*, f_l, f_l^*, f_r, f_r^*\}$. Each of these can be written in the form of the monomial listed above. 

Suppose that result holds for any word $x'$ with $l(x')<n$. Consider $x ( \neq 0)$ with $l(x)=n$. Now $x = x_1x_2 \cdots x_{n-1}x_n$ and $x_1\cdots x_{n-1}$ is a monomial of the form listed above by the induction hypothesis. Now since $x \neq 0$ then $x_{n-1}$ cannot be a right Fermionic edge, a left Fermionic ghost edge, Fermionic edge or Fermionic ghost edge. This implies that $x_{n-1}$ is a Fermionic vertex or that $x_{n-1}$ is an edge or ghost edge such that $r(x_{n-1}) \in B$. Note also that if $x_n$ is an edge or ghost edge then $s(x_n) \in B$. So $x_n$ can be a Bosonic vertex, a Fermionic vertex, an edge $e$ such that $s(e) \in B$ or a ghost edge $e^*$ such that $r(e) \in B$. 

\medskip

\noindent {\bf Case 1:} Let $x_{n-1}$ be a Fermionic vertex. Then $x_{n-1} = w_{i,l}$ such that $x= x_1 \cdots x_{n-l-1}w^ix_n$ where $w^i = w_{i,1} \cdots w_{i,l}$. Note that $l$ can equal $1$ here. 
\begin{itemize}
\item Let $x_n = v \in B$, then we can write $x = x_1 \cdots x_{n-l-1}w^i v = x_1 \cdots x_{n-l-1}vw^i$ and $v = r(x_{n-l-1})$ therefore $x = x_1\cdots x_{n-l-1}w^i$ which is a monomial as above. 
\item Let $x_n = w \in F$, then we can write $x = x_1 \cdots x_{n-l-1}w^iw \neq 0$ so $w \neq w_{i,j}$ for any $j \in \{1,...,l\}$ thus we can write $x = x_1 \cdots x_{n-l-1}w'^i$ such that $w'^i = w_{i,1}\cdots w_{i,l}w$ and $x$ is a monomial as written above. 
\item Let $x_n$ be an edge or ghost edge. Note $s(x_n) \in B$ and therefore, we have $x = x_1 \cdots x_{n-l-1}w^i s(x_n)x_n = x_1 \cdots x_{n-l-1} r(x_{n-l-1})s(x_n)w^i x_n$ therefore $r(x_{n-l-1}) = s(x_n)$. Therefore $x$ is a monomial as above. 
\end{itemize}
{\bf Case 2:} Let $x_{n-1}$ be an edge or a ghost edge such that $r(x_{n-1}) \in B$. Then we have two cases, $x_{n-1} = e^*_{i,l}$  or $x_{n-1} = e_{i,l}$. 

\medskip

\noindent Suppose $x_{n-1} = e^*_{i,l}$ then $x = x_1 \cdots x_{n-l-1} \beta_i^*x_n$ where $\beta_i^* = e^*_{i,1} \cdots e^*_{i,l}$. 
\begin{itemize}
\item Let $x_n = w \in F$, then $x = x_1 \cdots x_{n-l-1} \beta_i^* x_n$ which is a monomial as above. 
\item Let $x_n = v \in B$, then $x = x_1 \cdots x_{n-l-1} \beta_i^* v = x_1 \cdots x_{n-l-1} \beta_i^*$ as $v = r(\beta_i^*)$ since $x \neq 0$ and $x$ is therefore a monomial as above. 
\item  Let $x_n = e$ an edge in $E$, now $s(e), r(e) \in B$ since else $x =0$. Therefore $e = e_{i,l}$ and $x = x_1 \cdots x_{n-l-1} \beta_i^{'*}$ where $\beta_i^{'*} = e_{i,1}^* \cdots e_{i,l-1}^*$ then $x$ is as above. 
\item Let $x_n = e^*$ a ghost edge in $E$ now $s(e^*) \in B$ and $s(e^*) = r(e_{i,l}^*)$ so that $x = x_1 \cdots x_{n-l-1}\beta_i^{'*}$ where $\beta_i^{'*} = e_{i,1}^* \cdots e_{i,l}^*e^*$ and $x$ is a monomial as above. 
\end{itemize}
Suppose $x_{n-1} = e_{i,l}$, then $x = x_1 \cdots x_{n-l-1} \alpha_i x_n$ where $\alpha_i = e_{i,1} \cdots e_{i,l}$.  
\begin{itemize}
\item Let $x_n = w \in F$ then $x = x_1 \cdots x_{n-l-1} \alpha_i w$ which is a monomial as above. 
\item Let $x_n = v \in B$ then $r(\alpha_i) = v$ else $x =0$ therefore $x = x_1 \cdots x_{n-l-1}\alpha_i v = x_1 \cdots x_{n-l-1} \alpha_i$ which is a monomial as above. 
\item Let $x_n = e$ be an edge in $E$ then $s(e) = r(\alpha_i) \in B$ since $x \neq 0$ therefore $x = x_1 \cdots x_{n-l-1}\alpha'_i$ where $\alpha'_i = e_{i,1} \cdots e_{i,l}e$. Note that $r(e)$ could be either Fermionic or Bosonic, so $x$ is a monomial as above. 
\item Let $x_n = e^*$ be a ghost edge in $E$. Now $r(e) = s(\alpha_i) \in B$ else $x =0$ so we can write $x = x_1 \cdots x_{n-l-1}\alpha_i\beta_i^*$ where $\beta_i^* = e^*$. Note that $s(e)$ can be either Fermionic or Bosonic and therefore $x$ is a monomial as above. 
\end{itemize}

This shows that any monomial of length $n$ is of the form stated in the beginning and thus the proof is complete by induction. 
\end{proof} 

\noindent Let $x = x_1x_2 \cdots x_n$ be a monomial. We call $x$ to be {\it left Fermionic} if $s(x_1) \in F$ and $r(x_n) \in B$. We call $x$ to be {\it right Fermionic} if $r(x_n) \in F$ and $s(x_1) \in B$. We call a monomial to be a {\it Fermionic monomial} if $s(x_1) \in F$  and $r(x_n) \in F$.  Similarly, we say that a monomial is a {\it Bosonic monomial} if $s(x_1) \in B$ and $r(x_n) \in B$. \\

Note that these monomials look very similar to the monomials in Leavitt path algebras of the separated graphs $(E, C)$, namely a directed graph $E$ together with a family $C$ that gives partitions of the set of edges departing from each vertex of $E$ defined by Ara and Goodearl in \cite{AG}. The Leavitt path algebra of the separated graph $(E, C)$ is the $K$-algebra $L_K(E, C)$ with generators $\{v, e, e^*: v\in E^0, e\in E^1\}$ subject to the following relations: (1) $vv'=\delta_{v, v'}v$, for all $v, v'\in E^0$, (2) $s(e)e=e=er(e)$ for all $e\in E^1$, (3) $r(e)e^*=e^*=e^*s(e)$ for all $e\in E^1$, (4) $e^*e'=\delta_{e, e'} r(e)$, for all $e, e'\in X$, $X\in C$, and (4) $v=\Sigma_{e\in X} ee^*$ for every finite set $X\in C_v, v\in E^0$.         

Ara and Goodearl \cite{AG} have shown that if $(E,C)$ is a separated graph, then the set of reduced paths of the form below form a K-basis for $L_K(E,C)$. 
\[ \alpha_1\beta_1^* \alpha_2\beta_2^* \cdots \alpha_n \beta_n^*, \ \ \alpha_i, \beta_i \in Path(E),  n \geq 1, \] 
where $\beta_i$ and $\alpha_{i+1}$ are C-separated paths for each $i \in 1, ..., r-1$.

Although the monomials are very similar, there are two main things that distinguish these algebras from each other. First, there is no way to define a separation structure on the graph $E$ so that $L_K(E,C) \subseteq SL_K(E)$ if $C_v \neq \{s^{-1}(v)\}$. If we define $C_v = \{e \in s^{-1}(v): r(e) \in B \} \cup \{f \in s^{-1}(v): r(f) \in F\}$ which is the most logical way to define this, we run into problems. Consider $f$ such that $f \in s^{-1}(v)$ and $r(f) \in F$ and $e$ such that $e \in s^{-1}(v)$ and $r(e) \in B$. Then in $L_K(E,C)$, $f^*e \neq 0$ but $f^*e =0$ in $SL_K(E)$. Similarly, recall that in $L_K(E,C)$ if $w_1,w_2 \in F$ then $w_1 w_2 = - w_2w_1$. There is no element in $L_K(E,C)$ that anticommutes like the Fermionic vertices in $SL_K(E)$. Therefore, $SL_K(E) \not \subseteq L_K(E,C)$. Therefore while their monomials are similar, there is no connection between Leavitt path superalgebras and Leavitt path algebras of separated graphs except that they are both extensions of $L_K(E)$. 
\bigskip

\section{Properties of Leavitt path superalgebras}

\bigskip

\noindent In this section we will establish some basic properties of Leavitt path superalgebras. We begin by characterizing the Jacobson radical of Leavitt path superalgebras. Recall that the Leavitt path algebras have trivial Jacobson radical. 

\begin{theorem}
Let $E$ be a connected graph. Then the Jacobson radical of a Leavitt path superalgebra $SL_K(E)$ is the ideal generated by elements $x$ of the following form
\[ x = \sum_{i=1}^n k_ix^f_i + \sum_{j=1}^m k_jx^l_j + \sum_{s=1}^t k_sx^r_s \] 
where each $x^f_i$ is a Fermionic monomial, $x^l_j$ is a left Fermionic monomial and $x^r_s$ is a right Fermionic monomial. 
\end{theorem}

\begin{proof}
Denote the ideal generated by elements $x$ of the form $ x = \sum_{i=1}^n k_ix^f_i + \sum_{j=1}^m k_jx^l_j + \sum_{s=1}^t k_sx^r_s $ by $I$ and the Jacobson radical of $SL_K(E)$ by $J$. We will show that $I=J$. First, note that for every $x \in I$, $x^3 = 0$ therefore $I$ is a nil ideal and hence $I \subseteq J$. Since $SL_K(E)/I = SL_K(E')$ where $E'$ is a completely Bosonic graph, and in the view of the observation that for a completely Bosonic graph $E'$, the Leavitt path superalgebra $SL_K(E')$ coincides with the Leavitt path algebra $L_K(E')$, we have that $SL_K(E)/I= L_K(E')$ and therefore $J(SL_K(E)/I) = 0$. This gives us $J \subseteq I$ and consequently, we have $I = J$. This completes the proof. 
\end{proof}

Note that in a Leavitt path algebra the Jacobson radical is zero. The following example shows a graph such that the Leavitt path superalgebra cannot be represented by a Leavitt path algebra, since the Jacobson radical is non-zero. 

\begin{example}
Let $E$ be the graph below with $v$ a bosonic vertex and $w$ a fermionic vertex.
\begin{center}
\begin{tikzpicture}
\node[vertex, label={below: $v$}] (v) at (-1,0) {};
\node[vertex, label={below: $w$}] (w) at (1,0) {};
\path[->] (v) edge node[below]{$e$} (w); 
\end{tikzpicture} 
\end{center} 
Note that the edge $e \in J(SL_K(E))$ since $ex = 0$ for all $x \in SL_K(E)$ so that it is an annihilator of every simple right R-module. Therefore $J(SL_K(E))$ is non-zero and $SL_K(E)$ cannot be represented as a Leavitt path algebra. 
\end{example} 

It is not difficult to see that the Leavitt path algebra $L$ of a connected graph $E$ over a field $K$ is commutative if and only if the graph $E$ is either a single vertex or consists of a single vertex $v$ and an edge $e$ which is a loop at $v$. Next, we proceed to characterize Leavitt path superalgebras that are supercommutative.

\begin{theorem}
The Leavitt path superalgebra $SL_K(E)$ is a supercommutative superalgebra if and only if $E$ is a collection of Fermionic edges and Bosonic $R_1$ graphs along with solitary vertices. 
\end{theorem} 

\begin{proof} 
First, notice that all pairs of vertices satisfy the supercommutativity condition by definition. Now suppose that $e$ is a non-Fermionic edge (that is, $e$ is Bosonic, left Fermionic or right Fermionic) then, without loss of generality, $s(e) \not \in F$ and so $s(e)e = e$. But, $es(e) \neq e$ unless $r(e) = s(e)$. Therefore the only non-Fermionic edges are Bosonic edges where $r(e) = s(e) \in B$. Suppose $f$ is a Fermionic edge then $fx = 0 = xf$ for any $x \in SL_K(E)$ therefore we can have Fermionic edges. Now suppose that $e_1$ and $e_2$ are non-Fermionic edges so that $r(e_1) = s(e_1)$, $r(e_2) = s(e_2)$. Then if $r(e_1) \neq r(e_2)$ then $e_1 e_2 = 0 = e_2 e_1$. But if $r(e_1) = r(e_2)$, then $e_1 e_2 \neq e_2e_1$. Therefore our graph $E$ consists of Fermionic edges and vertices as well as Bosonic $R_1$ graphs. 
\end{proof}

\noindent We define a (not necessarily unital) ring $R$ to be {\it von Neumann regular} if for every $x \in R$, there exists $y \in R$ such that $x = xyx$. It is known that the Leavitt path algebra $L_K(E)$ over an arbitrary graph $E$ is von Neumann regular if and only if $E$ is acyclic \cite{AR}. A graded ring $R = \bigoplus R_\gamma$ is said to be {\it graded von Neumann regular} if for every homogeneous element $x \in R$ there exists $y \in R$ such that $xyx = x$. Note that $y$ can be chosen to be a homogeneous element. In \cite{Hazrat} it was shown that for any graph $E$, $L_K(E)$ is a graded von Neumann regular ring. We seek to mimic these results for Leavitt path superalgebras. 

\begin{proposition} 
The Leavitt path superalgebra $SL_K(E)$ is von Neumann regular if and only if $E$ has no Fermionic vertices and $E$ is acyclic. 
\end{proposition} 

\begin{proof} 
Let us assume that $E$ has a Fermionic vertex $w$. Now $w$ is a homogeneous element of $SL_K(E)$. Assume that there exists $y \in SL_K(E)$ such that $wyw = w$. Now looking at the parity of both sides of this equation implies that $y$ must have parity $1$. This implies that $y = \sum_{i=1}^n x_i$ where each $x_i$ is a Fermionic vertex, a left Fermionic path or a right Fermionic path. 
\begin{itemize} 
\item {\bf Case 1:} Let $x_i = w' \in F$. Then $ww'w = - www' = 0 \neq w$, a contradiction. 
\item {\bf Case 2:} Let $x_i$ be a left Fermionic path $\alpha$. Then $w\alpha w = 0w = 0 \neq w$, a contradiction
\item {\bf Case 3:} Let $x_i$ be a right Fermionic path $\beta$. Then $w\beta w = w0 = 0 \neq w$, a contradiction. 
\end{itemize} 
Therefore there does not exist an element $y \in SL_K(E)$ such that $wyw = w$. This shows that if $E$ has Fermionic vertices, then $SL_K(E)$ is not von Neumann regular. Now, if $E$ has no Fermionic vertices, then $SL_K(E) = L_K(E)$, and so by \cite{AR} it follows that in this case $SL_K(E)$ is von Neumann regular if and only if $E$ is acyclic. Thus we have that the Leavitt path superalgebra $SL_K(E)$ is von Neumann regular if and only if $E$ has no Fermionic vertices and $E$ is acyclic. 
\end{proof} 
 
\noindent The next question to then consider is if $E$ has no Fermionic vertices (that is, $E$ is a completely Bosonic graph) is $SL_K(E) (= L_K(E))$, a graded von Neumann regular ring under the grading of $SL_K(E)$? This is not true in general, since under the grading of $SL_K(E)$, every element of $L_K(E)$ is homogeneous.

\bigskip

\section{Growth of Leavitt path superalgebras} 

\bigskip 

\noindent In \cite{AAJZ}, the authors obtained the Gr\"{o}bner-Shirsov basis for Leavitt path algebras and proved that the Leavitt path algebra $L_K(E)$ over a finite graph $E$ has polynomially bounded growth if and only if two distinct cycles of $E$ do not have a common vertex. Furthermore, they proved that if $d_1$ is the maximal length of a chain of cycles in $E$, and $d_2$ is the maximal length of a chain of cycles with an exit, then $GK dim(L_K(E)) = max(2d_1-1, 2d_2)$. Structure of Leavitt path algebras of polynomial growth has been further studied in \cite{AAJZ2}.

In this section we will obtain analogues of these results for Leavitt path superalgebras. Our arguments are very close adaptations of those in \cite{AAJZ}. We will be using the following generalization of the Diamond lemma from \cite{Bergman}. 

\begin{theorem}\cite{Bergman} \label{Diamond}
Let $S$ be a reduction system for a free associative algebra $k\langle X \rangle$, and $\leq$  a semigroup partial ordering on $\langle X \rangle$, compatible with $S$, and having descending chain condition. Then the following conditions are equivalent:
\begin{enumerate}[(a)] 
\item All ambiguities of $S$ are resolvable. 
\item All ambiguities of $S$ are resolvable relative to $\leq$. 
\item All elements of $k\langle X\rangle$ are reduction-unique under $S$. 
\item A set of representatives in $k\langle X\rangle$ for the elements of the algebra $R = k\langle X \rangle/ I$ determined by the generators $X$ and the relations $W_\sigma = f_ \sigma$ ($\sigma \in S$) is given by the k-submodule $k\langle X \rangle_{irr}$ spanned by the $S$-irreducible monomials of $\langle X \rangle$. 
\end{enumerate} 
\end{theorem}

\noindent For each vertex $v \in B$ which is not a sink, choose a Bosonic edge $\gamma(v)$ such that $s(\gamma(v)) = v$ and $r(\gamma(v)) \in B$. We will refer to this edge as special. In other words, we fix a function $\gamma: B \backslash \{\text{sinks}\} \rightarrow E$ such that $s(\gamma(v)) = v$ for arbitrary $v \in B \backslash \{\text{sinks}\}$. 

\begin{theorem}\label{Basis}
The following elements form a basis of the Leavitt path superalgebra $SL_K(E)$: 
\begin{enumerate} 
\item $v \in V = B \cup F$, 
\item $p$, $p^*$ where $p = e_1e_2 \cdots e_n$ is a path in $E$, with $r(e_i) = s(e_{i+1}) \in B$ for $i \in  \{1, ..., n-1\}$,
\item $pq^*$ where $p$, $q$ are paths as in $(ii)$ and $r(p) = r(q) \in B$ and the last edges in $p$ and $q$ are distinct or equal but not special, and
\item irreducible products of the elements above. 
\end{enumerate} 
\end{theorem} 

\begin{proof}
We first introduce a well-ordering of generators of $SL_K(E)$. First, we order the elements of $v_i \in B \subset E^0$ and $w_i \in F \subset E^0$, such that $v_i < v_j$ if $i < j$ and $w_i < w_j$ if $i < j$. We further order the vertices as $w > v$ if  $w \in F$ and $v \in B$. Let $e, f \in E^1$ then $e < f$ if $s(e) < s(f)$.  Now we order the edges coming out of the same Bosonic vertex, if $v \in B$, let $\{e_1, ..., e_k : s(e_i) = v, r(e_i) \in B$ with $e_k = \gamma(v)\}$ and let $\{f_1, ..., f_k: s(f_i) = v, r(f_i) \in F \}$ then $e_1 < e_2 < \cdots < e_k = \gamma(v) < f_1 < f_2 < \cdots < f_k$. We can then arbitrarily order the edges $\{ f_i: s(f_i) \in F\}$. Also, if $e^*, f^* \in (E^1)^*$ then $e^* < f^*$ if $e < f$. Finally, $v < w < e < f^*$ if $v \in B$, $w \in F$, $e \in E^1$ and $f^* \in (E^1)^*$. We then must modify our relations so that the relations are compatible with our well ordering that is if $W_\sigma = f_\sigma$ is a relation, then $W_\sigma > f_\sigma$.  
\begin{enumerate}
\item \begin{enumerate} 
\item $v_iv_j = \delta_{ij} v_i$ for $v_i, v_j \in B$. 
\item $w_iw_j = - w_jw_i$ for $i>j$ $w_i, w_j \in F$. 
\item $wv = vw$, for $w \in F$ and $v \in B$. 
\end{enumerate} 
\item \begin{enumerate}
\item $s(e)e = e$ if $s(e) \in B$, $r(e)e^* = e^*$ if $r(e) \in B$ 
\item $s(e)e =0$ if $s(e) \in F$ , $r(e)e^* = 0$ if $r(e) \in F$
\item $er(e) = e$ if $r(e) \in B$, $e^*s(e) = e^*$ if $s(e) \in B$
\item $er(e) = 0$ if $r(e) \in F$, $e^*s(e) = 0$ if $s(e) \in F$	
\end{enumerate} 
\item \begin{enumerate} 
\item $e^*e = r(e)$ if $r(e), s(e) \in B$.
\item $e^*e = 0$ if $r(e) \in F$ or $s(e) \in F$.  
\item $e^*f = 0$ if $e \neq f$ with both $e$ and $f$ being Bosonic edges
\end{enumerate}
\item If $v \in B$ is any regular vertex, then 
$\gamma(v)\gamma(v)^* = v - \sum_{e \in E^1: s(e) = v, r(e) \in B, e \neq \gamma(v)} ee^*$ 
\end{enumerate} 

Note that the only changes in these were in (1) and (4). We also add in a relation that is compatible with the relations above which helps simplify elements.   
\begin{enumerate}
\item[(5)] If $v \neq s(e)$ then $ve =0$ and $e^*v = 0$. If $v \neq r(e)$ then $ev = 0$ and $e^*v = 0$. 
\end{enumerate} 

\noindent We let $S$ be the modified relations for $SL_K(E)$ shown above and $X = E^0 \cup E^1 \cup (E^1)^*$. The well-ordering on $\langle X \rangle$ is compatible with $S$ and has the descending chain condition, that is, every set of elements has a smallest element. Note that by the definition of $S$ above, all ambiguities of $S$ are resolvable and therefore the condition (d) of Theorem \ref{Diamond} gives us that $SL_K(E) = K \langle X \rangle / I$ is given by the $k$-submodule $k\langle X \rangle$ spanned by the $S$-irreducible monomials of $\langle X \rangle$. Hence, we conclude that the basis of $SL_K(E)$ consists precisely of the irreducible elements stated above. 
\end{proof} 

\noindent Now, in order to characterize the polynomially bounded growth Leavitt path superalgebra $SL_K(E)$ over a finite graph $E$, we begin with a lemma. Before stating the lemma, let us fix a notation which we will frequently use. For any cycle $C$, we will denote by $V(C)$, the set of all vertices appearing in the cycle $C$.  

\begin{lemma} 
If two distinct Bosonic cycles have a common vertex, then the Leavitt path superalgebra $SL_K(E)$ has exponential growth. 
\end{lemma} 

\begin{proof}
Let $C_1$ and $C_2$ be two distinct Bosonic cycles in $E$ such that $V(C_1) \cap V(C_2) \neq \emptyset$. Let $C_1$ be the cycle $v_1 v_2 \ldots v_m v_1$ where for $1\le i\le m-1$, $e_i$ is the arrow from $v_i$ to $v_{i+1}$ and $e_m$ is the arrow from $v_m$ to $v_1$. Similarly, let $C_2$ be the cycle $w_1 w_2 \ldots w_n w_1$ where for $1\le i\le n-1$, $f_i$ is the arrow from $w_i$ to $w_{i+1}$ and $f_n$ is the arrow from $w_n$ to $w_1$. Let $p = e_1e_2\cdots e_m$ and $q = f_1f_2 \cdots f_n$ then $p$ and $q$ generate a free subalgebra in $SL_K(E)$. Therefore $SL_K(E)$ has exponential growth. 
\end{proof} 

For two cycles $C'$ and $C''$, we write $C' \Rightarrow C''$, if there exists a Bosonic path that starts in $C'$ and ends in $C''$. A sequence of distinct Bosonic cycles $C_1, ..., C_k$ is called a chain of length $k$ if $C_1 \Rightarrow C_2 \Rightarrow \cdots \Rightarrow C_k$. The chain is said to have an exit if the cycle $C_k$ has an exit. Let $d_1$ be the maximal length of a chain of cycles in $E$, and let $d_2$ be the maximal length of a chain of cycles with an exit. Clearly $d_2 \leq d_1$. 

\begin{theorem}
Let $E$ be a finite graph. Then the Leavitt path superalgebra $SL_K(E)$ has a polynomial growth if and only if two distinct Bosonic cycles of $E$ do not have  a common vertex. 
\end{theorem}

\begin{proof} 
We assume that no two distinct Bosonic cycles have a common vertex and we note that the generating set of $SL_K(E)$ is $X = E^0 \cup E^1 \cup (E^1)^*$. Let $E'$ be the set of edges that do not belong to any cycle. Let $P'$ be the set of all non-fermionic paths that are composed of edges from $E'$. Then an arbitrary path from $P'$ never arrives to the same vertex twice. Hence, $|P'| < \infty$. By the basis theorem above, the space $Span(X^n)$ is spanned by the elements of the following types: 
\begin{enumerate} 
\item a vertex $v \in B$, or $w^i = w_{i,1} \cdots w_{i,j}$ such that each $w_{i,k} \in F$ and $w_{i,k} \neq w_{i,l}$ if $k \neq l$. 
\item a path $p = p'_1p_1 p_2'p_2p_k p'_{k+1}$ where $p_i$ is a path on a Bosonic cycle $C_i$, $1 \leq i \leq k$, $C_1 \Rightarrow \cdots \Rightarrow C_k$ is a chain $p_i' \in P'$, $p_i'$ is a Bosonic path, and $length(p) \leq n$.  
\item $p^*$, where $p$ is a path of the type $(2)$, 
\item $pq^*$, where $p = p_1'p_1p_2'p_2 \cdots p_k p_{k+1}'$, $q = q_1'q_1q_2'\cdots q_2q_{s+1}$; $p_i$, $q_j$ are paths on Bosonic cycles $C_i$, $D_j$ repectively and $C_1 \Rightarrow \cdots C_k$, $D_1 \Rightarrow \cdots \Rightarrow D_s$ are chains; $p_i', q_j' \in P'$, $length(p) + length(q) \leq n$ with $r(p) = r(q)$. 
\item $x^1w^1x^2 \cdots w^l x^{l+1}$ where $x^i$ looks like $(2), (3), (4)$ above and $w^i =w_{i,1}w_{i,2}\cdots w_{i,j}$ where each $w_{i,k}$ is distinct and $r(x_i) = s(x_{i+1})$. \\
\end{enumerate} 

\noindent We first look at the vertices in $(1)$. Now for $v \in B$ we clearly have $n^0$ growth since our graph $E$ is finite. Now we consider the strings of Fermionic vertices $w^j$. Let $w^j$ have length $n$ then we can write $w^j$ in $|F|(|F|-1) \cdots (|F| - (n-1))$ ways which is a polynomial in $n$ with some degree $l_j$. Therefore these strings have polynomial growth $n^{l_j}$.  

\noindent We next estimate the number of paths of length $n$ of the type $(2)$. First, fix a chain $C_1 \Rightarrow C_2 \Rightarrow \cdots \Rightarrow C_k$. Let $C =v_1 v_2 \ldots v_m v_1$ be a cycle where for $1\le i\le m-1$, $e_i$ is the arrow from $v_i$ to $v_{i+1}$ and $e_m$ is the arrow from $v_m$ to $v_1$. Let $P_C= e_1 \cdots e_m$. For a given cycles there are $m$ such paths depending upon the choice of starting point $v_1$.  \\

\noindent Let $|V(C_i)| = m_i$ and let $P_{C_i}$ be any one of the $m_i$ paths described above. Then an arbitrary path on $C_i$ can be represented as $u' P_{c_i}^l u''$ where $length(u'), length(u'') \leq m_i -1$. Hence every path of the type $(2)$ which corresponds to the chain $C_1 \Rightarrow C_2 \Rightarrow \cdots \Rightarrow C_k$ can be represented as $p'_1 u'_1 P_{C_1}^{l_1} u_1'' \cdots p_k'u'_k P_{C_k}^{l_k} u_k'' p'_{k+1}$, where $p_i' \in P'$ and $length(u_i'), length(u_i'') \leq m_i -1$. Clearly, $m_1l_1 + \cdots + m_k l_k \leq n$. We show that the number of such paths is $\preceq n^k \leq n^{d_1}$ by induction. Suppose $k = 1$, then $m_1l_1 \leq n = n^1$ and we are done. Assume that for a cycle of length $k-1$ there are $n^{k-1}$ choices for paths of length $n$ on the chain. Now consider a chain of length $k$, there are $n^{k-1}$ choices for paths of length $n$ on the first $k-1$ cycles and there are $m_1l_1 + m_2 l_2 + \cdots m_kl_k$ choices for paths using all $k$ cycles in the chain therefore the number of choices for paths of length $n$ on a chain of length $k$ are 
\begin{align*}
&\preceq n^{k-1} + m_1l_1 + m_2l_2 + \cdots m_kl_k \\
& \leq n^{k-1} + n \\
& \leq n^k
\end{align*}  
Therefore the number of paths of length $n$ on a chain of length $k$ is $\preceq n^k \leq n^{d_1}$. Now the number of elements of type $(3)$ is the same as in type $(2)$. Therefore the elements of type $(2)$ and type $(3)$ have polynomial growth.  

Now we consider elements of type $(4)$. Consider elements of length $\leq n$ of the type $pq^*$, $r(p) = r(q) \in B$; the path $p$ passes through the cycles of the chain $C_1 \Rightarrow C_2 \Rightarrow \cdots \Rightarrow C_k$ on the way, the path $q$ passes through the cycles of the chain $D_1 \Rightarrow D_2 \Rightarrow \cdots \Rightarrow D_s$ on the way. So $p = p_1'p_2p_2' \cdots p_k p'_{k+1}$ where $p_i' \in P_i'$, each $p_i$ is a paht on the cycle $C_i$. Similarly, $q = q_1'q_1 q_2' \cdots q_sq'_{s+1}$. Arguing as above, we see that for fixed chains $C_1 \Rightarrow C_2 \Rightarrow \cdots \Rightarrow C_k$ and $D_1 \Rightarrow D_2 \Rightarrow \cdots \Rightarrow D_s$, the number of such paths is $\preceq n^{k+s} \leq n^{2d_1}$. Therefore, elements of type $(4)$ have polynomial growth.

Finally we consider elements of type $(5)$, let $x = x^0 w^1 x^1 w^2 x^2 \cdots x^{l-1}w^l x^l$ be an element of length $ \leq n$. Let each $x^i$ have length $n_i$. Then by the above, $x^i$ can be written in $n_i^{k_i}$ ways where $k_i$ is an integer. Now let each $w^j$ have length $m_j$. Also by above $w^j$ can be written in $m_j ^{l_j}$ ways. Now $x$ has length $ \leq n$ so that $\sum n_i + \sum m_j  \leq n$, in particular $n_i \leq n$ and $m_j \leq n$. Therefore the number of ways to write $x$ with length $ \leq n$ is $ < \Pi_i n_i^{d_i} \Pi_j m_j^{l_j}$. This is a polynomial in $n_i$ and $m_j$ with degree $\sum_i k_i + \sum_j l_j$, therefore the growth of $x$ can be expressed as $n^{\sum d_i + \sum d_j}$ and is a polynomial. This shows that every element of $SL_K(E)$ has a polynomial growth and hence we conclude that $SL_K(E)$ has a polynomial growth. 
\end{proof}

\bigskip

\noindent {\bf Acknowledgement.} The authors would like to thank Michele Cirafici and K. M. Rangaswamy for many helpful comments and suggestions.   

\bigskip

\bigskip


\bigskip

\bigskip

\end{document}